\documentclass[a4paper,english, 12pt]{article}
\usepackage{babel} 
\usepackage[ansinew]{inputenc}
\usepackage{amsmath, amssymb, amsthm}
\usepackage{nicefrac}
\usepackage[arrow, matrix, curve]{xy}
\usepackage[nospace,noadjust]{cite}
\usepackage{graphics}
\usepackage{graphicx}
\usepackage[lmargin=3.00cm,rmargin=3.00cm,tmargin=3cm,bmargin=3cm]{geometry}
\usepackage{tikz}

\newcommand{\R}{\mathcal{R}}
\newcommand{\Reals}{\mathbb{R}}

\newcommand{\Z}{\mathbb{Z}}

\newcommand{\Q}{\mathcal{Q}}

\newcommand{\C}{\mathbb{C}}
\newcommand{\D}{\mathcal{D}}
\newcommand{\B}{\mathcal{B}}
\newcommand{\F}{\mathcal{F}}
\newcommand{\K}{\mathcal{K}}

\renewcommand{\H}{\mathcal{H}}
\renewcommand{\P}{\mathcal{P}}
\newcommand{\E}{\mathcal{E}}

\newcommand{\X}{\mathfrak{X}}

\newcommand{\const}{\textnormal{const}}
\newcommand{\id}{\textnormal{id}}

\newcommand{\colim}{\textnormal{colim}}

\newcommand{\into}{\hookrightarrow}
\newcommand{\tr}{\textnormal{tr}}

\newcommand{\onto}{\twoheadrightarrow}

\newcommand{\M}{\mathcal{M}}
\renewcommand{\L}{\mathcal{L}}

\renewcommand{\S}{\mathcal{S}}

\newcommand{\ot}{\leftarrow}
\newcommand{\hocolim}{\textnormal{hocolim}}

\newcommand{\hofib}{\textnormal{hofib}}

\newcommand{\W}{\mathcal{W}}
\newcommand{\simp}{\textnormal{simp}}
\newcommand{\cone}{\textnormal{cone}}
\newcommand{\curlyC}{\mathcal{C}}

\renewcommand{\Re}{\textnormal{Re}}

\renewcommand{\tilde}{\widetilde}
\renewcommand{\bar}{\overline}

\newtheorem{Lemma}{Lemma}[section]
\newtheorem{Theorem}[Lemma]{Theorem}
\newtheorem*{Theorem*}{Theorem}
\newtheorem{Proposition}[Lemma]{Proposition}
\newtheorem{Corollary}[Lemma]{Corollary}

\theoremstyle{definition}

\newtheorem{Definition}[Lemma]{Definition}

\theoremstyle{remark}
\newtheorem{Remark}[Lemma]{Remark}
\newtheorem{Observation}[Lemma]{Observation}

\title{Comparison of Higher Smooth Torsion}
\author{Christopher Ohrt}

\begin{document}

\maketitle

\begin{abstract} By explicitly comparing constructions, we prove that the higher torsion invariants of smooth bundles defined by Igusa and Klein \cite{Igusa2} via Morse theory agree with the higher torsion invariants defined by Badzioch, Dorabiala, Dwyer, Weiss, and Williams using homotopy theoretical methods (\cite{zbMATH05551153} and \cite{zbMATH02115779}).
\end{abstract}  

\tableofcontents

\section{Introduction}

Higher torsion aims to provide an invariant classifying differential structures that can be put on a smooth bundle $F\into E\onto B.$ Many different approaches have been taken to define higher torsion invariants (\cite{Igusa2}, \cite{0793.19002}, \cite{0837.58028}, \cite{zbMATH02115779}, \cite{1071.58025}). In general a torsion invariant will take as argument the bundle $E$ and a finite, fiber-wise local system $\F\to E$ (where the action of $\pi_1 B$ on $H_*(F;\F)$ is sufficiently trivial) and produce a cohomology class $\tau(E;\F)\in H^{2k}(B;\mathbb{R}).$ To compare different torsion invariants, Igusa developed a system of axioms for such objects and used it to classify higher torsion in the case where $\F$ is trivial \cite{Igusa1}. The author expanded this to the more general ``twisted'' case of arbitrary (finite) $\F$ and finite $\pi_1 B$ \cite{Axioms}. Based on these axioms many comparisons were made \cite{zbMATH05848700}, \cite{Goette}.

Two particular constructions of higher torsion invariants were given by Igusa and Klein \cite{Igusa2} and Badzioch, Dorabiala, Dwyer, Weiss, and Williams \cite{zbMATH05551153}.  Given a parametrized, generalized Morse function $f:E\to \mathbb{R}$ the former construct their invariant $\tau^{IK}$ by carefully analyzing the evolution of critical Morse points in the fiber $F_x$ as $x\in B$ varies and explicitly give a map in K-Theory along which they pull back the Borel regulator \cite{MR0387496} to get the homology class $\tau_{IK}(E;\F)\in H^{2k}(B;\mathbb{R}).$ In contrast, the latter use homotopy theory to find a lift of the Becker-Gottlieb transfer $p^!:B\to \Omega^{\infty}\Sigma^\infty(E_+)$ into the fiber of the composition 
		$$\Omega^\infty\Sigma^\infty(E_+)\xrightarrow{\textnormal{assembly}} A(E)\xrightarrow{\textnormal{linearization}_\F}K(\C)$$
	and also pull back the Borel regulator to define their smooth torsion homology class $\tau_{sm}(E;\F)\in H^{2k}(B;\Reals).$ Originally this was done for trivial local systems and then the author extended it to finite local systems \cite{twistedsmooth}. 
	
The explicit nature of Igusa-Klein torsion opens it up for calculations on $S^1$-bundles and many more calculations can be accessed via the axioms \cite{Igusa2}. The natural definition of smooth torsions makes this a very intuitive and universal tool, but is prohibitive to calculations. To the author's knowledge, there are currently no non-trivial results. Using the axioms, Badzioch, Dorabiala, Klein, and Williams showed that smooth torsion is a multiple of Igusa-Klein torsion related by a non-zero factor if restricted to trivial local systems. The exact value of said factor was not determined \cite{zbMATH05848700}. In \cite{twistedsmooth} the author shows that for general (finite) local systems, smooth torsion still satisfies almost all axioms, with the obstruction being the continuity axiom which requires an explicit calculation of smooth torsion for $S^1$-bundles.

In this paper we will prove directly:

\begin{Theorem}\label{mainthm} Igusa-Klein and smooth torsion agree, whenever they are defined:
	$$\tau^{IK}(E;\F)=\tau^{sm}(E;\F).$$
\end{Theorem}

As an immediate consequence we get many specific values for smooth torsion from the calculation of Igusa-Klein torsion \cite{Igusa2} such as 

\begin{Corollary} For a linear $S^1$-bundle $S(\xi)\to B$ associated to the complex line bundle $\xi\to B$ with local system $\F_\zeta$ given by a root of unity $\zeta\in \mathbb{C}^*$ we have
\begin{eqnarray*}\tau_{k}^{sm}(S(\xi); \F_{\zeta})=-{n^k}L_{k+1}(\zeta)ch_{2k}(\xi).
\end{eqnarray*}
The function $L_{k+1}$ is the real polylogarithm defined by
	$$L_{k+1}(z)=\Re\left(\frac 1 {i^k} \L_{k+1}(z)\right),$$
where $\L_{k+1}$ is the complex polylogarithm
		$$\L_{k+1}(z)= \sum_{m=1}^\infty \frac{z^m}{m^{k+1}}.$$
Consequently smooth torsion satisfies the continuity axiom.
\end{Corollary}

Besides this we can extend the definition of Igusa-Klein torsion: The original construction only works if the fundamental group $\pi_1B$ acts trivially on $H_*(F;\F),$ whereas for smooth torsion it is enough to say that said action is unipotent, i.e. there is a filtration $0=V_0\subset V_1\subset\ldots\subset V_n=H_*(F;\F)$ where $\pi_1 B$ acts trivially on the quotients $V_{i+1}/V_i.$ Now we can define $\tau^{IK}:=\tau^{sm}$ if the action is unipotent and since $\tau^{sm}$ satisfies all axioms this behaves naturally. In particular if $\pi_1 B$ is finite, higher torsion is defined for any finite local system $\F.$ Goodwillie and Igusa have recently announced to be able to make this extension of the definition explicitly as well. 

Our strategy for the proof is as follows: Let $Wh_\F(E)$ be the fiber of the composition $\Omega^\infty\Sigma^\infty(E_+)\to A(E)\stackrel{\F}{\longrightarrow}K(\C)$ mentioned above. Then both the Igusa-Klein torsion and the smooth torsion are defined by pulling back a certain homology class along a map $\tau^{IK}:B\to Wh_\F(E)$ and $\tau^{sm}:B\to Wh_\F(E)$ respectively. Both constructions use very different models for the involved spaces, however. Inspired by the unpublished manuscript \cite{IgusaWaldhausen}, we give an explicit unifying model and use it to compare the two torsion maps.

\textbf{Outline:} In sections \ref{smoothtorsion} and \ref{IKtorsion} we will recall the definitions of smooth and Igusa-Klein torsion respectively. In Section \ref{comparison} we will provide the unifying model and compare the torsions.

\textbf{Acknoledgements:} The author wants to thank Kiyoshi Igusa for several very helpful conversations and for pointing towards the expansion categories, which proved be the key to the comparison result.

\section{The higher smooth torsion map} 
\label{smoothtorsion}

\subsection{The manifold approach}

This section repeats the constructions made in the beginning of \cite{zbMATH05551153} and \cite{twistedsmooth}. Let $X$ be a compact manifold. We will define a model for $\Omega^\infty\Sigma^\infty X_+$ which (by abusing notation a bit) we will call $Q(X_+).$ It will be constructed as the direct limit under stabilization of the Waldhausen K-theory spaces of certain categories of partitions. We will refrain from giving details as they can be found in \cite{zbMATH05551153}.

A partition of $X\times I$ is a (not necessarily smooth and possibly with corners) codimension 0 submanifold $M\subset X\times I$ that represents the lower half of a division of the interval in two parts and is somewhat standard around the boundary and on the lower third. Part of the data is also a vector field transversal to the boundary which can be used to smoothen the partition. The set  $\P_k(X)$ consists of partitions of $X\times I$ parametrized over $\Delta^k.$ These fit together in a simplicial set $\P_\bullet(X).$ There is a stabilization map $\P_\bullet (X)\to \P_\bullet (X\times I)$ defined by putting the non-trivial part of a partition of $X\times I$ into the middle third (of the second interval) to get a partition of $(X\times I)\times I.$ We note that there is a partial monoid structure on $\P_\bullet (X)$ where we add two partitions of $X\times I$ if they do not share any non-trivial parts. Stabilization now provides a monoid structure on $\colim_n \P_\bullet(X\times I^n).$ 

The sets $\P_k(X)$ can also be viewed as partially ordered sets  by inclusion, and hence as categories. So we can apply the Waldhausen $\S_\bullet$-construction  (or rather the Thomason variant thereof) \cite{zbMATH03927168} to get bisimplicial categories $\S_\bullet \P_\bullet (X).$ Recall that the objects of the category $\S_n\P_0(X)$ are $(n+1)$-tuples of partitions $(M_i)_{i=0}^n$ with
	$$M_0\subset M_1\subset\ldots\subset M_n,$$
together with identifications of any subquotients. Here $M_0$ is required to be the initial partition $X\times [0,\frac 1 3].$ Note that the space $|\S_\bullet \P_\bullet(X)| $ is endowed with a canonical base point.

By stabilization we get a space 
	$$Q(X_+):=\Omega\hocolim_n|\S_\bullet \P_\bullet(X\times I^n)|\simeq \Omega^\infty\Sigma^\infty X_+.$$
The weak equivalence on the right is rather intricate and was shown by Waldhausen in \cite{zbMATH03958257} and \cite{zbMATH04163135}.

Recall that the algebraic K-Theory of the space $X$ is defined as 
	$$A(X)=\Omega |\S_\bullet \R^{hf}(X)|$$
where $\R^{hf}(X)$ is the Waldhausen category of homotopy finite retractive spaces over $X$ \cite{zbMATH03927168}. By ``thickening up'' this model for the algebraic K-theory of spaces one can define a map 
	$$\alpha:Q(X_+)\to A(X)$$
that roughly takes a partition over $X$ and views it as a retractive space over $X.$ This map is a model for the assembly. See \cite{zbMATH05551153} for details.

\begin{Remark} Since the assembly map $Q(X_+)\to A(X)$ has a homotopy left-inverse \cite{zbMATH03958257}, we won't need to fully understand this map, but merely know that it exists. For details compare the proof of proposition \ref{InverseTrick}.
\end{Remark}

\begin{Remark} We will often use the simplicially enriched model $A(X^{\Delta^\bullet})$ for $A(X).$ The objects of $\R^{hf}(X^{\Delta^n})$ are $\Delta^n$-families of retractive spaces over $X,$ which can also be viewed as retractive spaces $Y$ over $X\times \Delta^n$ together with a projection $Y\to \Delta^n$ fitting in the following commutative diagram
	$$\xymatrix{
		Y\ar[dr]\ar@<0.5ex>[rr]&&X\times \Delta^n\ar[dl]^{\textnormal{pr}_2}\ar@<0.5ex>[ll]\\ & \Delta^n
		}
	$$
	We have $A(X)\simeq A(X^{\Delta^\bullet})$ given by the inclusion of zero simplices. This simplicial enrichment is similar to the one used in the definition of $Q(X_+),$ so we can view the assembly map as $\alpha:Q(X_+)\to A(X^{\Delta^\bullet}).$
\end{Remark}

\begin{Remark} \label{mapsinto} The following helps greatly in defining maps into $Q(X_+)$ (and $A(X)$). Recall that there is a natural map
	$$|\S_1\mathcal{W}|\times \Delta^1\to |\S_\bullet\mathcal{W}|$$
for any Waldhausen category $\mathcal{W}$ given by the inclusion of the 1-skeleton in the $\S_\bullet$ direction \cite{zbMATH03927168}. After taking the adjoint this gives a map 
	$$|\S_1\mathcal{W}|\to K\mathcal{W}.$$
Hence it is always enough to define a functor $\mathcal{C}\to \mathcal{W}\cong\S_1\mathcal{W}$ to get a map $|\mathcal{C}|\to K\mathcal{W}$ for any small category $\mathcal{C}.$
\end{Remark}

\subsection{The transfer map}

Let $S_\bullet(B)$ be the simplicial category of simplices $\sigma:\Delta^\bullet\to B$ with no non-trivial morphisms. Clearly, we have $|S_\bullet (B)|\simeq B.$ 

Let $E\to B$ be a smooth bundle. There is a transfer map 	
	$$p^!_A:|S_\bullet (B)|\to A(E^{\Delta^\bullet})$$
given by the functor that sends a simplex $\sigma:\Delta^n\to B$ to the retractive space
	$$E\times \Delta^n\sqcup \sigma^*E\leftrightarrows E\times \Delta^n.$$

One can explicitly construct a lift $p^!:|S_\bullet(B)|\to Q(E_+)$ such that $\alpha\circ p^!\sim p^!_A$ are homotopic where $\alpha:Q(X_+)\to A(E^{\Delta^\bullet})$ \cite{zbMATH05551153}.

\begin{Proposition}\label{InverseTrick} The map $p^!:|S_\bullet(B)|\to Q(E_+)$ has the homotopy type of the Gottlieb-Becker transfer $p_{BG}:B\to \Omega^\infty \Sigma^\infty X_+.$
\end{Proposition}

\begin{proof} We adopt the proof from \cite{zbMATH05551153}.  Let $\tr:A(X)\to\Omega^\infty \Sigma^\infty X_+$ be Waldhausen's trace map \cite{zbMATH03958257}, a right inverse to $\alpha$ with $\alpha\circ \tr\sim \id_{Q(E_+)}.$ It is know that the composition $\tr\circ p^!_A\sim p_{BG}.$ So we have
	$$p_{BG}\sim \tr\circ p^!_A\sim \tr \circ\alpha\circ p^!\sim p^!.$$
\end{proof}

\begin{Remark} Because of the existence of the Waldhausen trace map as in the proof above, we do not need to explicitly understand the transfer map $p^!$ but rather only $p^!_A.$
\end{Remark}

\subsection{Linearization}

We still follow \cite{zbMATH05551153} closely to define linearization maps. Let $R$ be a ring and let $Ch^{hf}(R)$ be the Waldhausen category of homotopy finitely dominated chain complexes of projective $R$-modules. Recall that the Waldhausen K-theory of this category is just a model for the algebraic K-theory $K(R)$ of $R$ \cite{zbMATH03927168}.

Now let $X$ be a compact manifold and $\F$ a local system of $R$-modules on $X.$ Then we get a functor
	$$\R^{hf}(X)\to Ch^{hf}(R)$$
by sending a retractive space $X\to Y\to X$ to the relative singular chain complex $C_*(Y,X;\F).$ This induces a linearization map
	$$\lambda^\R_\F:A(X)\to K(R)$$
and if we compose with the assembly $\alpha:Q(X_+)\to A(X)$ we get a map 
	$$\lambda_\F:Q(X_+)\to K(R).$$
	
Let $E\to B$ be a bundle of compact manifolds and let $\F$ be a local system of $R$-modules on $E.$ Similarly to before we can define a functor
	$$S_\bullet(B)\to wCh^{hf}(R)$$ 
(the $w$ indicates that we are only looking at quasi-isomorphisms as morphisms.) In particular, this functor sends a simplex $\sigma:\Delta^k\to B$ to the chain complex $C_*(\sigma^*E, \F).$ Using Remark \ref{mapsinto} this gives rise to a map
	$$c_\F: |S_\bullet(B)|\to K(R).$$

\begin{Theorem} Let $E\to B$ be a bundle of compact manifolds, $R$ a ring, and $\F$ a local system of $R$-modules on $E.$ Then there is a preferred homotopy which makes the following diagram commute:
	$$\xymatrix{
		&	Q(E_+)\ar[d]^{\lambda_\F}\\
		|S_\bullet(B)|\ar[ur]^{p^!}\ar[r]_{c_\F}	&	K(R)}$$
	The homotopy is induced by the isomorphism $H_*(\sigma^*E;\F)\cong H_*(E\sqcup \sigma^* E, E;\F).$
\end{Theorem}

\begin{proof} The composition $\lambda_\F\circ p^!\sim \lambda^\R_\F\circ p^!_A$ sends the simplex $\sigma:\Delta^n\to B$ to the chain complex
	$$C_*(E \sqcup \sigma^E,E;\F)$$
which is homotopy equivalent to 
	$$C_*(\sigma^*E;\F)$$
which is the image of $\sigma$ under $c_\F.$
\end{proof}

This will be the starting point for us to define smooth parametrized torsion.

\subsection{Unreduced and reduced smooth parametrized torsion}

This section is where we depart slightly from \cite{zbMATH05551153} in that our results will be a little bit more general than there. This also appears in \cite{twistedsmooth}. The idea here is that if we can show that the map $c_\F:|S_\bullet(B)|\to K(R)$ is homotopic to the constant map with value the 0 complex $0\in K(R)$ then we get a lift 
	$$|S_\bullet(B)|\to \hofib\left(Q(E_+)\stackrel{\lambda_\F}{\longrightarrow} K(R)\right)_0=:Wh_\F(E)$$
where we call the codomain the Whitehead space of $E.$ This will not always be the case, but the following condition is almost sufficient:

\begin{Definition}\label{unipotent} Let $E\to B$ be as before and let $\F$ be a complex local system on $E.$ Let $B$ be connected, $b_0\in B$ be the basepoint, and let $F$ be the fiber over $b_0.$ We say $\pi_0 B$ acts unipotently on $H_*(F;\F)$ if there exists a filtration of $H_*(F;\F)$ by $\pi_1 B$ submodules 
	$$0=V_0(F)\subset \ldots \subset V_k(F)=H_*(F;\F)$$
such that $\pi_1 B$ acts trivially on the quotients $V_i(F)/V_{i-1}(F).$
\end{Definition}

\begin{Theorem}\label{contraction} Let $E\to B$ be a bundle, $B$ path-connected, $b_0\in B$ the basepoint, $F_{b_0}$ the fiber over the basepoint,  $\F\to E$ a complex local system such that $\pi_1 B$ acts unipotently on $H_*(F,\F).$ Then there exists a preferred homotopy from the map $c_\F:|S_\bullet(B)|\to K(\C)$ to the constant map with value the complex $H_*(F_{b_0},\F)\in K(\C)$ (with trivial boundary maps).
\end{Theorem}

\begin{proof} This can be found in \cite{twistedsmooth} or adapted from \cite{zbMATH05551153}.
\end{proof}

\begin{Definition} Let $p:E\to B$ be a compact manifold bundle with $B$ connected. Let $F_{b_0}$ be the fiber over the basepoint and let $\F$ be a unipotent complex local system over $E.$ We view the homology complex $H_*(F_{b_0};\F)$ as an element in $K(\C)$ and we define the unreduced Whitehead space
	$$Wh_\F(E,b_0):=\hofib(Q(E_+)\stackrel {\lambda_\F}\to K(\C))_{H_*(F_{b_0};\F)}.$$
The unreduced smooth torsion of $p$ is the map $\tilde \tau_\F:|S_\bullet(B)|\to Wh_\F(E,b_0)$ determined by the transfer $p^!$ and the homotopy $\omega_\F.$
\end{Definition}

We want to make this independent of the basepoint choice. The answer is the reduced torsion:

\begin{Definition}\label{reducedtorsion} For a compact manifold bundle $p:E\to B$ with base point $b_0\in B$  and unipotent complex local system $\F$ on $E$ we define the Whitehead space 
	$$Wh_\F(E):=\hofib(Q(E_+)\to K(\C))_0.$$
The reduced smooth torsion $\tau_\F(p)$ is the map $|S_\bullet(B)|\to Wh_\F(E)$ obtained from $p^!$ by subtracting the element $p^!(b_0)\in Q(E_+)$ from the map $p^!$ and the path $\omega_{\F|{b_0}\times I}$ from the contracting homotopy $\omega_\F.$
\end{Definition}

\begin{Remark} So far this only defines the torsion map. The cohomological torsion $\tau_\F^{k}(p)\in H^{2k}(B;\Reals)$ is defined in the following way: If $\F$ is trivial, consider the final map $Q(E_+)\to Q(S^0).$ This lifts to a map on Whitehead spaces $Wh(E)\to Wh(*).$ Since the middle term in the homotopy fibration 
	$$Wh(*)\to Q(S^0)\to K(\C)$$
is rationally contractible, and the cohomology class $b_{k}$ (the Borel regulator) of $K(\C)$  therefore gives a cohomology class $b_k\in Wh(*).$ We then pull this back along the composition
	$$|S_\bullet(B)|\to Wh(E)\to Wh(*)$$
to get the cohomological torsion.

In the case where $\F$ is non-trivial, one replaces the point $*$ with the ``equivariant point'' $BG$ (where $G$ is a finite group). To do so, a manifold approximation to $BG$ is needed. See \cite{twistedsmooth}.
\end{Remark}

\renewcommand{\R}{\mathbb{R}}

\section{Igusa-Klein Torsion}
\label{IKtorsion}

In this section we will define the Igusa-Klein torsion. We will first give an intuitive description for the construction for $S^1$-bundles $E\to B.$ This will motivate the explicit definitions of the categorical models for the Whitehead space. Then we will generalize these models to accommodate the definition of the Igusa-Klein torsion for any smooth bundle $E\to B.$ Lastly, we will explain why the models involved have the correct homotopy type.

\subsection{Torsion of $S^1$-bundles}

Let $E\to B$ be an $S^1$-bundles and $\F$ a local system on $E$ that is completely determined on the fiber (thus $\F$ can just be viewed as a root of unity). Assume that $\F$ is non-trivial so that the singular chain complex $C_*(S^1,\F)$ is acyclic. Now choose a fiber-wise generalized Morse function $f:E\to \R$ (by \cite{Igusa2} this is a contractible choice). This means that on every fiber $S^1\cong E_x$ over $x\in B$ the function $f$ restricts to  either a proper Morse function or a function that may only have critical points that in local coordinates look like 
	$$f(x_1)=c\pm x^2\quad \textnormal{(critical point of order 2)}$$
or 
	$$f(x_1)=c\pm x^3\quad \textnormal{(critical point of order 3)}.$$
The set of points $x\in B$ over which the generalized Morse function $f$ gives critical points of order 3 on $E_x$ forms a codimension 1 submanifold of $B$ called the bifurcation set.

Now imagine two points $x,y\in B$ outside off the bifurcation set and a path connecting $x$ and $y$ by crossing the bifurcation set. This means that as we move from $x$ to $y$ either two critical points of $E_x$ come together in a critical point of degree 3 and cancel each other out or two critical points in $E_y$ are created from a critical point of degree 3. See figure \ref{Figure} for an example of an $S^1$-bundle over the interval $I.$  The idea of Igusa-Klein torsion is to codify and track information about critical points and the Morse complexes of $E_x$ as $x$ varies in the base space $B$ and use this information to define a torsion invariant. 

\begin{figure}
\centering

\definecolor{Black}{HTML}{000000}
\begin{tikzpicture}[xscale=.018cm, yscale=.018cm]

[draw=Black]

\draw  plot[smooth cycle, tension=.7] coordinates {(-3.5,-2) (-1.5,-1.5) (-1,0.5) (-1.5,2.5) (-2.5,2.5) (-3.5,1) (-4.5,2) (-5,2) (-5.5,1) (-5.5,-1)};
\node at (-5,2.5) {$x_3$};

\draw  plot[smooth cycle, tension=.7] coordinates {(-10.5,-1.5) (-10,0) (-10,1.5) (-10.5,2) (-11.5,2) (-12,3) (-12.5,4.5) (-13.5,4.5) (-14,3) (-14.5,2) (-15.5,3.5) (-16.5,1) (-16,-1.5) (-13,-2)};

\draw  plot[smooth cycle, tension=.7] coordinates {(-21,-1.5) (-20,0.5) (-21,3) (-22,1.5)  (-23,3) (-23.5,4.5) (-24.5,4.5) (-25,3) (-25.5,1.5)  (-26.5,3.5) (-27.5,1.5) (-27,-1)  (-24,-2) };
\draw (-1,-4.5) -- ++(-26,0);
\node at (-26.5,-4) {$B=I$};
\node at (-24,-5) {$b_1$};
\node at (-13,-5) {$b_2$};
\node at (-3.5,-5) {$b_3$};
\node at (-3,-2.5) {$y_3$};
\node at (-2,3) {$x_2$};
\node at (-3.5,0.5) {$y_2$};
\node at (-10.5,2.5) {$x_1=y_1$};
\node at (-13,5) {$x_2$};
\node at (-14.5,1.5) {$y_2$};
\node at (-15.5,4) {$x_3$};
\node at (-13,-2.5) {$y_3$};
\node at (-1.5,-1.5) {$\bullet$};
\node at (-10.5,-1.5) {$\bullet$};
\node at (-21,-1.5) {$\bullet$};
\node at (-21,3.5) {$x_1$};
\node at (-22,1) {$y_1$};
\node at (-24,5) {$x_2$};
\node at (-25.5,1) {$y_2$};
\node at (-26.5,4) {$x_3$};
\node at (-24,-2.5) {$y_3$};
\end{tikzpicture}
\caption{A maximum $x_1$ and minimum $y_1$ coming together at a birth death point.} 
\label{Figure}

\end{figure}
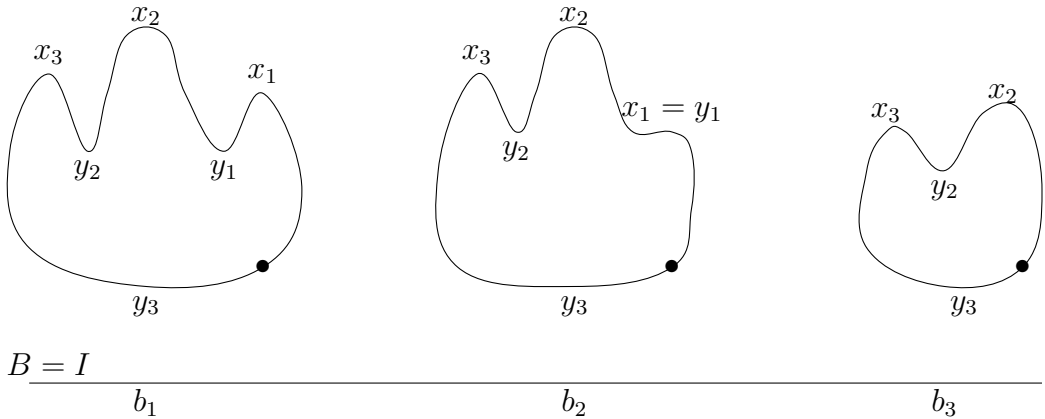

To do so let $S_\bullet(B)$ again be the simplicial category of simplices in $B$ with no morphisms. We will encode the above information as a functor $S_\bullet(B)\to \W_\bullet(\C),$ where $|\W_\bullet(\C)|$ is the Whitehead space. Following \cite{Igusa2} we will first give an explicit model for the simplicial category $\W_\bullet(C)$ and later show that it has the desired homotopy type. Guided by the varying Morse complexes of the fibers $E_x$ over $x\in B$, we see that the main feature this category should have is that its 0-simplices are Morse complexes, whereas its 1-simplices enable a connection between the Morse-complexes on different sides of the bifurcation set of $B.$ Here is the formal definition:

\begin{Definition}[2.1.1 and 2.1.7 of \cite{Igusa2}] The simplicial category $\W_\bullet(\C,n)$ is given by the following:
\begin{itemize}
	\item Its objects in degree $p$ are pairs $(C_*,P)$ where
	\begin{itemize}
		\item  $P=P_0\sqcup P_1$ is a graded partially ordered set where $P_0$ and $P_1$ have exactly $n$ elements with grading $0$ or $1$ respectively,
		\item $C_*$ is a $p+1$-tuple of upper triangular (in the partial ordering) of isomorphisms $f(i):\C^{P_1}\to\C^{P_0}$  viewed as acyclic chain complexes $0\to \C^{P_1}\to \C^{P_0}$ together with chain isomorphisms $E(i,j):f(i)\to f(j)$ for $i\leq j,$ homotopies $E(j,k)E(i,j)\to E(i,k),$ and higher homotopies. 
	\end{itemize}
	\item morphisms $(C_*,P)\to (C_*',P')$ are given by a closed bijection $P\to P'$ and a monomial chain morphisms over it (see \cite{Igusa2} for details).
	\item Face and degeneracy maps are given by deleting and repeating terms in the object tuples
\end{itemize}
\end{Definition}

The idea behind this definition is that a Morse function $f:S^1\to \R$ with exactly $2n$ critical points of degree 0 and 1 (ordered by the Morse function)
gives a 0-simplex of $\W(\C,n)$ via forming the Morse complex. This is not enough to treat our bundles $S^1\to E\to B$ as we do not expect every fiber to have the same amount of critical points. Hence, we need to stabilize:

\begin{Definition}[2.5.1-2.5.3 in \cite{Igusa2}] We make the following definitions:
	\begin{itemize}
		\item Let $(C_*,P)$ be an object of $\W_p(\C,n).$ An expansion pair is a pair of elements $x^-\in P_0$ and $x^+\in P_1$ such that $x^-<x^+,$ they are unrelated to any other elements, and $f(x^+)=gx^-$ with $g\in U(1).$
		\item Let $(C_*,P)\in \W_p(\C,n)$ and $(C_*',P')\in \W_p(\C,n').$ An expansion $(C_*,P)\to(C_*',P')$ is a degree 0 poset embedding $P\to P'$ such that $P'\backslash P$ is a union of expansion pairs, together with a chain monomorphism over said inclusion.
		\item The simplicial category $\W_\bullet (\C)$ has as objects any objects of $\W_\bullet (\C,n)$ for any $n$ and morphism the morphisms in $\W_\bullet(\C,n)$ together with the expansions. This can also be considered a bicategory.
	\end{itemize}
\end{Definition}

We are now ready to define the functor $S_\bullet(B)\to \W_\bullet(\C).$ Let $S^1\to E\to B$ be a smooth bundle and let $\zeta\in U(1)$ be a non-trivial root of unity (playing the role of a local system $\F\to E$). Choose a generalized parametrized Morse function $f:E\to \R.$ It is clear that such a map should assign any 0-simplex $x\in B$ the Morse complex of $E_x$ with coefficients $\zeta.$ However, the main difficulty arises in how to define the functor on 1-simplices that cross the bifurcation set as the two ends of such a simplex do not have to have the same amount of critical points. The solution to this are ``ghost'' points:

After a pair of critical points of the fiber of $E\to B$ meet and cancel over the bifurcation set, they remain detectable (in a neighborhood) as an inflection point on which the second derivative of $f:E_x\to \R$ vanishes. We call these points ghosts and we can choose our simplices small enough such that over any simplex where there is a pair of critical points dissolving into a ghost the corresponding ghost does not dissolve.  We do not change the homotopy type of $|S_\bullet(B)|\simeq B$ by only considering such small simplices. We will only call an inflection point in the fiber over a given point a ghost over a given small simplex, if there is a point within that simplex where the ghost develops into two critical points.  By the  discussion, over any point in a given simplex the number of critical points plus twice the number of ghosts will be the same. With this we define the functor $S_\bullet(B)\to \W(\C)$ on 1-simplices by sending the path $\gamma:\Delta^1\to B$ to $(C_*,P)$ where 
	$$P=\{\textnormal{critical points of degree 0, ghosts}\}\sqcup\{\textnormal{critical points of degree 1, ghosts}\}$$
and $C_*$ is given by the Morse complexes of $E_{\gamma(0)}$ and $E_{\gamma(1)}$ (with expansion pairs for ghosts) together with a chain isomorphism given by their connection. The functor can be defined similarly on higher simplices. For details see \cite{Igusa2}.

\begin{Remark} Upon careful inspection, one notices that the so defined $S_\bullet(B)\to \W_\bullet(\C)$ is not simplicial. However, this can be alleviated by introducing a weak equivalence $\curlyC_\bullet(B)\to S_\bullet(B)$ and a simplicial functor $\curlyC_\bullet(B)\to \W_\bullet(\C).$ This is done explicitly in \cite{Igusa2}. By abuse of notation, we will continue to write $S_\bullet(B)$ instead of $\curlyC_\bullet(B).$
\end{Remark}

\subsection{The Whitehead Category}

In the previous section we defined the Igusa-Klein torsion for any $S^1$-bundle $E\to B$ by explicitly constructing a functor $S_\bullet(B)\to \W(\C).$ Now we turn our attention to the more general case: Let $F\to E\to B$ be a smooth manifold bundle. We again wish to define its Igusa-Klein torsion. Let $\F\to E$ be a local system (trivial on $B$) and choose a generalized parametrized Morse function $f:E\to B.$ The Morse complexes of $E_x\cong M$ will not be concentrated in two degrees alone anymore, so to define the torsion functor we need a target category encoding general chain complexes as $0$-simplices, together with homotopies and higher homotopies as higher simplices:

\begin{Definition}[3.1.1, 3.2, 3.6 in \cite{Igusa2}] Let $n_*=(n_1,\ldots,n_k)$ be a tuple of natural numbers. The simplicial category $Wh_\bullet(\C,n_*)$ is defined as follows:
	\begin{itemize} 
		\item An object in degree $q$ is a pair $(C_*,P)$ where
			\begin{itemize}
				\item $P$ is a partially ordered graded set with $n_i$ elements in each degree $i$
				\item $C_*$ is a $q+1$-tuple of chain complexes where the $i$th entry of the $l$th chain complex ($1\leq l\leq q+1$)  is $C_*(l)_i=\C^{P_i}$ (the boundary maps can be different for each $l$), together with homotopies and higher homotopies connecting the entries in the $q+1$-tuple as if they were corners of a $q$-simplex.
			\end{itemize}
		\item Morphism are given by closed bijections $P\to P'$ and sufficiently coherent collections of chain morphisms over them.
	\end{itemize}
This can again be stabilized via expansion pairs to obtain the category $Wh_\bullet(\C).$ The full subcategory $Wh^h_\bullet(\C)$ is given by only considering objects comprised of acyclic chain complexes.
\end{Definition}

Notice that this was completely analogous to the definition of $\W_\bullet(\C)$ and contains the former as a subcategory.

\begin{Proposition}[\cite{Igusa2}] Let $M\to E\to B$ be a smooth fiber bundle and $\F\to E$ a finite local system as above such that the singular complex $C_*(M,\F)$ is acyclic. Then the (contractible) choice of a generalized parametrized Morse function $f:E\to \R$ defines a functor $S_\bullet(B)\to Wh^h(\C)$ analogously to the previously defined functor $S_\bullet(B)\to \W(\C)$ for $S^1$-bundles.
\end{Proposition} 

\begin{Remark} Only the category $Wh^h_\bullet(\C)$ has the desired homotopy type of the Whitehead space and not $Wh_\bullet(\C).$ Hence we only define the torsion functor for acyclic fibers and not more generally. We will consider a slightly more general case in subsection \ref{extendingIK}.
\end{Remark}

\subsubsection{Filtered chain complexes} 

The above construction proves to be somewhat unwieldy as the simplicial structure of $Wh_\bullet(\C)$ is quite complicated. Instead we will use the so called ``multiple mapping cylinder'' to turn an object of $Wh_\bullet(\C)$ - that is a system of chain complexes and higher homotopies - into a single filtered chain complex. We first define the latter:

\begin{Definition}[4.1.1 and 4.1.2 in \cite{Igusa2}] Let $P$ be a (partially ordered, graded) set.
	\begin{itemize}
		\item A $\Lambda P$-module is a $\C$-vector space $M$ together with subspaces $M^A$ for all $A\subset P$ such that
			\begin{itemize}
				\item $M^P=M$ and
				\item $M^{A\cap B}=M^A\cap M^B.$
			\end{itemize}
			The $\Lambda P$-modules naturally form a category.
		\item A $\Lambda P$-filtered $\C$-complex $(E,\lambda)$ is a chain complex in the category of $\Lambda P$-modules together with a cohomology class
			$$\lambda_A(x)\in H^{\deg x}(E^{A\sqcup\{x\}},E^A;\C)$$ 
		for all pairs $A\subset P$ and $x\in P$ such that
		\begin{itemize}
			\item $E^\emptyset=0$
			\item $E^{\{x\}}+E^A=E^{A\cup \{x\}}$
			\item $H_{\deg x}(E^{A\sqcup \{x\}},E^A;\C)\cong \C$ via the map induced by $\lambda_A(x)$ and this relative homology vanishes in all other degrees.
			\item The cohomology classes $\lambda_A(x)$ are compatible.
		\end{itemize}
	\end{itemize}
\end{Definition}

\begin{Remark} One can think of a $\Lambda P$-filtered chain complex as a chain complex together with a basis $P_i$ (elements of $P$ in degree $i$) for its $i$th homology for all $i.$
\end{Remark}

\begin{Definition}[4.1.3 \cite{Igusa2}] There is a multiple mapping cylinder construction turning an object $(C_*,P)\in Wh_q(\C)$ into a $\Lambda P$-filtered chain complex $Z_q(C_*).$ This is done by assembling all the homotopy information from $C_*$ into a large chain complex. If $q=0$ then $Z_0(C_*)=C_*$ is itself already a filtered complex. The filtering on higher $q$'s is similar. 
\end{Definition}

\begin{Remark} As the name suggests, the idea of the multiple mapping cylinders is to take subsequent mapping cones: For example, let $(C_*,P)\in Wh_1(\C)$ be  a 1-simplex. That is it is completely represented by a chain complex homotopy equivalence $f:(\C^P)_0\to (\C^P)_1.$ To retain all the information of this map, while still condensing the structure into a single chain complex, we can take the mapping cone $\cone(f)\in Ch(\C),$ which naturally has the structure of a filtered chain complex over $P.$ Clearly we have homotopy equivalences $(\C^P)_0\simeq \cone(f)\simeq (\C^P)_1.$ For higher simplices one can subsequently form cones of the connecting maps and homotopies.
\end{Remark}

Next we define the category of filtered chain complexes. Recall that the classifying space $BU(1)$ can be viewed as the geometric realization of the simplicial set with $BU(1)_k=U(1)^k.$ Let $\xi$ be the universal line bundle over $BU(1).$ 

\begin{Definition}[5.2.2 in \cite{Igusa2}]\label{filteredchain} Let $FC(BU(1)_\bullet,\xi,n_*)$ be the following simplicial category:
	\begin{itemize}
		\item An object in degree $q$ is a triple $(E,P,\gamma),$ where
			\begin{itemize}
				\item $P$ is a partially ordered graded set with $n_i$ elements in each degree $i$
				\item $\gamma: P\to BU(1)_k=U(1)^k$ is a map of sets and
				\item $E$ is a $\Lambda P$-filtered chain complex with cohomology classes $\lambda_A(x)$ giving isomorphisms
					$$H_{\deg x}(E^{A\sqcup\{x\}},E^A)\cong \xi(\lambda(x))\cong \C.$$
			\end{itemize}
		\item A morphism is given by a closed bijection $\alpha:P\to P'$ and a sufficiently coherent chain complex morphism $E\to E'$ above it.
	\end{itemize}
This can be stabilized via extension pairs to a stable category $FC(BU(1)_\bullet,\xi).$ The full subcategory $FC^h(BU(1)_\bullet,\xi)$  is given by only considering acyclic chain complexes.
\end{Definition}

\begin{Proposition}[5.3.4 and 5.3.5 in \cite{Igusa2}] The multiple mapping cylinder construction gives weak homotopy equivalences
	$$Wh_\bullet(\C)\simeq FC(BU(1)_\bullet,\xi)$$
and
	$$Wh_\bullet^h(\C)\simeq FC^h(BU(1)_\bullet,\xi).$$
\end{Proposition}

\begin{Remark} The maps $\gamma:P\to U(1)^k$ for an object $(E,P,\gamma)$ in $FC(BU(1),\xi)$ are needed to encode morphism and expansion structures in $Wh(\C).$
\end{Remark}

\begin{Observation}\label{filteredtorsion} Let $M\to E\to B$ be a smooth bundle with acyclic local system $\F\to E$ and generalized Morse function $f:E\to \R.$ Instead of defining the torsion functor $S_\bullet(B)\to Wh(\C)$ we can also directly define the torsion functor $S_\bullet(B)\to FC(BU(1),\xi)$ by composing with the multiple mapping cylinder construction.
\end{Observation}

\begin{Remark} Definition \ref{filteredchain} can be generalized by replacing $BU(1)$ with any simplicial set $X$ with a functor $\xi:\simp X\to \textnormal{Vect}_\C.$ We call the resulting category $FC(X,\xi).$
\end{Remark}

\begin{Observation} Again let $M\to E\to B$ be a smooth bundle with acyclic local system $\F\to E$ and generalized Morse function $f:E\to \R.$  Then the local system defines a functor $\xi_\F:\simp E\to \textnormal{Vect}_\C$ and we can factorize the torsion functor through $FC^h(E,\xi_\F)\to FC^h(BU(1),\xi)$ to get

\begin{Definition} The construction above gives the Igusa-Klein torsion as a map
	$$S_\bullet(B)\to FC^h(E,\xi_\F).$$
\end{Definition}
\end{Observation}

\subsubsection{Extension of the definition}
\label{extendingIK}

At this point we only defined Igusa-Klein torsion for 1-dimensional acyclic local systems $\F\to E.$ We will briefly indicate how to remedy these shortcomings:

Analogously to our definition of $FC(BU(1),\xi)$ one can define $FC(BG,\xi)$ for any group $G$ together with a representation $G\to U(n).$ This creates a natural target for the torsion of any bundle $M\to E\to B$ wit acyclic but not necessarily 1-dimensional local system $\F\to E.$ Furthermore this can also be lifted to a torsion functor
	$$S_\bullet(B)\to FC^h(E,\xi_\F).$$

Now assume that $M\to E\to B$ is a smooth bundle with local system $\F\to E$ such that $\pi_1(B)$ acts trivially on the homologies of the fiber with coefficients $\F.$ However, $\F$ does not have to be acyclic anymore. As mentioned in \ref{filteredtorsion}, this still gives a functor (after choosing a Morse function $f:E\to \R$)
	$$S_\bullet(B)\to FC(E,\xi_\F)$$
but this functor will not factor through $FC^h$ anymore. However, if $\pi_1 B$ acts trivially on the homology of the fiber, we can after stabilization form the alternating mapping cone which will define a map
	$$S_\bullet(B)\to FC^h(E,\xi_\F).$$
We will take this map as the definition of the torsion in the non-acyclic case. Details are to be found in chapter 4.6 of \cite{Igusa2}.

\subsection{The Homotopy Type of $FC(E,\xi_\F)$ and the category $\Q(E)$}

We continue to summarize the constructions of \cite{Igusa2}. So far, we defined the Igusa-Klein torsion functor $S_\bullet(B)\to FC^h(E,\xi_\F)$ for smooth bundles $E\to B$ with local system $\F\to E,$ but so far we have not yet established that $FC^h(E,\xi_\F)$ has the correct homotopy type of the Whitehead space. Hence, in this section we show that $FC^h(E,\xi_\F)$ can be identified as the homotopy fiber of the composition $Q(E_+)\to A(E)\to K(\C).$ After this the cohomological Igusa-Klein torsion is defined just as the smooth torsion as pull-back of the Borel regulators. 

To identify $FC^h(E,\xi)$ as the homotopy fiber we use the Waldhausen fibration theorem \cite{zbMATH03927168}. We will introduce a category $\K$ with two kinds of weak equivalences ($w$-equivalences and $h$-equivalences). Then Waldhausen gives a homotopy fibration sequence (recall that we write $Kw(-)$ to indicate $\Omega|w\S_\bullet(-)|$ for the Waldhausen $\S_\bullet$-construction)
	$$Kw\K^h\to Kw\K\to Kh\K$$
and we identify $Kw\K^h\simeq FC^h(E,\xi_\F),$ $Kw\K\simeq Q(E_+),$ and $Kh\K\simeq K(\C).$ Before defining $\K,$ we will define a simpler category $\Q(E_+)$ with $K\Q(E_+)\simeq Q(E_+)$ which will facilitate the middle equivalence.

 \begin{Definition} Let $X_\bullet$ be a simplicial set, then the category $\Q(X_\bullet)$ of finite graded poset over $X_\bullet$ is defined as follows:
 	\begin{itemize}
		\item An object in degree $q$ is graded poset $P$ together with a map $\gamma:P\to X_k.$ As part of the data, there is a subset of identified expansion pairs $x_-, x_+$ in $P$ over the same point in $X_k$ with $\deg x_+=1+\deg x_-$ and $x_-<x_+.$
		\item Morphisms are pointed maps over $P.$
		\item A morphism $P\to Q$ is a cofibration if it is an order preserving monomorphism.
		\item A morphism $f:P\to Q$ is a $w$-equivalence if its kernel $\ker f=f^{-1}(*)$ is a union of expansion pairs and $f:P\backslash \ker f\to Q$ is a bijection.
	\end{itemize}
This forms a Waldhausen category.
\end{Definition}

\begin{Remark}\label{ordering} In \cite{Igusa2} Igusa considers a slightly different category $\Q_I({X_\bullet})$ in which there are no identified expansion pairs in the objects $P$ (and thereby $w$-equivalences are just bijections). 

Let $\Q_I^0(X_\bullet)$ be the subcategory of graded posets over $X_\bullet$ wit null ordering. Then the retraction $\Q_I(X_\bullet)\to \Q_I^0(\bullet)$ given by forgetting the ordering is a deformation retract. But points of different degrees don't interact in $\Q_I^0(X_\bullet),$ and we get
	$$\Q_I^0(X_\bullet)=\prod_{n\geq 0} \Q^0_{I,n}(X_\bullet),$$
where $\Q^0_n$ denotes the subcategory of isolated degree $n.$ Segal established $Kw\Q^0_n(X_\bullet)\simeq Q(|X_\bullet|)$ \cite{MR0331377}.

This means that $\Q_I(X_\bullet)$ does not have the correct homotopy type. However, we can take the nerve along elementary expansions to get $e_\bullet \Q_I(X_\bullet).$ An object of $e_k\Q_I(X_\bullet)$ is a sequence 
	$$P\to P\vee S_1\to\ldots P\vee S_k,$$
where the $S_k$ are increasing sets of expansion pairs. This mends together the different copies and we have
	$$Kwe_\bullet \Q_I(X_\bullet)\simeq Q(|X_\bullet|).$$
For a more detailed discussion compare \cite{Igusa2} 5.6.5 and following.
\end{Remark}

\begin{Observation} Notice that for any simplicial set $X$ we have
	$$|we_\bullet \Q_I(X_\bullet)|\simeq |w\Q(X_\bullet)|:$$
The left hand side is the geometric realization of the bicategory with objects graded posets over $X$, vertical morphisms expansions, and horizontal morphisms bijections. The right hand side is the geometric realization of the category with objects graded posets over $X$ and morphisms being compositions of bijections and collapses of expansion pairs. Consequently we have
	$$Kwe_\bullet \Q_I(X_\bullet)\simeq Kw\Q(X_\bullet).$$
\end{Observation}

\begin{Corollary} 
	$$Kw\Q(X_\bullet)\simeq Q(|X_\bullet |)$$
\end{Corollary}

\begin{Remark} As outlined in Remark \ref{ordering}, forgetting the orderings does not change the homotopy type, so from now on we will only work with graded sets. Furthermore we will continue to work with $\Q$ instead of $\Q_I.$ 
\end{Remark}

\begin{Definition} Let $X_\bullet$ be a simplicial set and $\xi:\simp X\to \textnormal{Vect}_\C$ a functor. 
	\begin{itemize} 
		\item The simplicial category $\K(X_\bullet,\xi)$ has similar objects  to $FC(X_\bullet, \xi)$, that is pairs $(E,P)$ where $P$ is a graded set and $E$ is a filtered chain complex over $P.$ Additionally there should be set of identified expansion pairs $x_-, x_+$ in $P$ and we demand that $E$ splits as
			$$E^{P\backslash\{x_-,x_+\}}\oplus E^{\{x_-,x_+\}},$$
		where the latter is given by $\C$ in degree $\deg x_-$ and $\deg x_+$ connected by the identity. 
		
		\item Morphisms $(f,\alpha):(E,P)\to (E',P')$ are again given by morphisms $\alpha$ of graded sets  (so morphisms in $\Q(X_\bullet)$) and chain morphisms $f$ over them.
		\item Cofibrations in $\K(X_\bullet,\xi)$ are cofibrations in $\Q(X_\bullet)$ covered by chain isomorphisms.
		\item A morphism $(f,\alpha)$ is a $w$-equivalence if $\alpha$ is a $w$-equivalence in $\Q(X_\bullet).$
		\item A morphism $(f,\alpha)$ is an $h$-equivalence if $f$ is a chain homotopy equivalence.
	\end{itemize}
\end{Definition}

\begin{Remark} Again in \cite{Igusa2} Igusa defines $\K_I(X_\bullet)$ without identification of the expansion pairs. As in Remark \ref{ordering}, one can then form $e_\bullet \K_I(X_\bullet)$ and this yield the same results as our $\K(X_\bullet).$
\end{Remark}

\begin{Observation} Since for any object $(E,P)$ of $\K(X_\bullet,\xi)$ the graded set $P$ acts as a ``homological basis'' for $E,$ it is clear that every $w$-equivalence is also an $h$-equivalence.
\end{Observation}

Based on this observation and Igusa's work showing that $\K_I(X_\bullet,\xi)$ has a mapping cylinder construction satisfying Waldhausen's cylinder axioms we get immediately:

\begin{Theorem}[Based on Waldhausen \cite{zbMATH03927168}]\label{fibthm} The sequence
	$$Kw\K(X_\bullet,\xi)^h\to Kw\K(X_\bullet,\xi)\to Kh\K(X_\bullet,\xi)$$
is a homotopy fibration with canonical contracting homotopy given by the unique natural transformation from the composition $w\K(X_\bullet,\xi)^h\to h\K(X_\bullet,\xi)$  to the constant functor on the final object.  Here the superscript $-^h$ indicates $h$-trivial objects.
\end{Theorem}

There is an obvious forgetful functor $ \K(X_\bullet,\xi)\to \Q(X_\bullet)$ which respects the $w$-equivalences.

\begin{Proposition}[\cite{Igusa2}] The induced functor map 
	$$Kw \K(X_\bullet,\xi)\to  Kw\Q(X_\bullet)$$
is a weak equivalence.
\end{Proposition}

Furthermore we get a functor 
	\begin{eqnarray}\label{map}
		\K(X_\bullet,\xi)\to Ch(\P^{hf}(\C))
	\end{eqnarray}
by forgetting the filtrations.

\begin{Proposition} The induced map
	$$Kh\K(X_\bullet,\xi)\to Kh Ch(\P^{hf}(\C))=K(\C)$$
is a weak equivalence.
\end{Proposition}

It is clear that
	$$w\K(X,\xi)^h\simeq FC^h(X_\bullet,\xi).$$
So the last ingredient to finish the characterization of $FC^h$ as a homotopy fiber is the following:

\begin{Proposition}[\cite{Igusa2}]  The natural map (given by Remark \ref{mapsinto})
	$$|w\K(X_\bullet,\xi)|\to Kw\K(X_\bullet,\xi)$$
is a weak equivalence.
\end{Proposition}

\renewcommand{\R}{\mathcal{R}}

\renewcommand{\X}[1][\bullet]{X^{\Delta^{#1}}}
\newcommand{\simE}[1][\bullet]{E^{\Delta^{#1}}}

\section{A combinatorical model for the Becker-Gottlieb transfer}
\label{comparison}

In the previous sections we defined the smooth and Igusa-Klein torsion of a smooth manifold bundle $E\to B$ with local system $\F\to E.$ Both were given as maps into the Whitehead space $Wh_\F(E):$ The smooth torsion was given as a lift of the Becker-Gottlieb transfer $p^!:|S_\bullet(B)|\to Q(E_+)$ whereas the Igusa-Klein torsion is directly constructed as a map $|S_\bullet(B)|\to Wh(E).$ Composition with the inclusion of the fiber will give a transfer map $p^!_{IK}:|S_\bullet(B)|\to Q(E_+),$ and --of course-- the Igusa-Klein torsion map is a lift of this. We will show that these two transfer maps have the same homotopy type, ultimately leading to a proof of Theorem \ref{mainthm}.

First of all, recall that the model used in the previous section is $|FC^h(E,\xi_\F)|\simeq Wh_\F(E).$ According to Observation \ref{filteredtorsion} the Igusa-Klein torsion 
	$$\tau_{IK}:|S_\bullet(B)|\to |FC^h(E,\xi_\F)|$$
is given by sending a simplex $\sigma:\Delta^k\to B$ to the pair $(P, C)$ where $P$ is the graded poset of critical and twice the birth-death points of a chosen fiber-wise generalized Morse-function on $E,$ the defining map is given by the map $P\times \Delta^k\to E$ obtained by lifting $\sigma:\Delta^k\to B$ to the level of critical and birth-death points (this is not necessarily injective as two critical points can meet in a birth-death point), and finally the filtered chain complex $C$ is given by taking the multiple mapping cylinder construction of the Morse-complexes over $\sigma$ with coefficients $\xi.$ The Igusa-Klein transfer
	$$p_{IK}:|S_\bullet(B)|\to Kw\K(E^{\Delta^\bullet},\xi_\F)$$
is then given in view of Theorem \ref{fibthm} as the composition of $\tau_{IK}$ with the maps
	$$|FC^h(E^{\Delta^\bullet},\xi_\F)^h|\simeq |w\K(E^{\Delta^{\bullet}},\xi_\F)^h|\xrightarrow{\textnormal{Remark \ref{mapsinto}}}Kw\K(E^{\Delta^{\bullet}},\xi_\F)^h\into Kw\K(E^{\Delta^{\bullet}},\xi_\F).$$
By Remark \ref{mapsinto} we can regard this composition as induced by the concrete functor described above.

We will give an alternate model for $Q(E_+)$ and use this to connect the Igusa-Klein transfer $p_{IK}$ to the Becker-Gottlieb transfer $p^!$  given in Section 2 as a lift 
	$$\xymatrix{
		& Q(E^{\Delta^\bullet} _+)\ar[d]\\
		|S_\bullet(B)|\ar[ur]^{p^!}\ar[r]^{p^1_A} & A(E^{\Delta^\bullet})
		}$$
where $p^!_A$ is given by sending $\sigma:\Delta^k\to B$ to the retractive space 
	$$E\times \Delta^k\sqcup \sigma^*E\times \Delta^k\to E\times \Delta^k.$$

\subsection{The expansion category}

We begin with the following definition:

\begin{Definition} Let $X$ be a topological space. We define the expansion category $\E(X^{\Delta^\bullet})$ as follows:
	\begin{itemize} 	
		\item An object in degree $k$ of $\E(X^{\Delta^k})$ is a triple $(P,Y,r),$ where
			\begin{itemize}
				\item $P$ is a graded poset over $X^{\Delta^k}$ with identified expansion pairs.
				\item $Y$ is a $k$-parameter family of relative cell complexes with cells indexed by $P.$ In particular
					$$Y=X\times \Delta^k\sqcup \bigsqcup_{p\in P} I^{\deg{p}}\times \Delta^k/\sim,$$
				where no cell is attached to a cell of equal or higher order. 
				\item Every expansion pair corresponds to two cells in canceling position, directly attached to $X.$
				\item $r:Y\to X\times \Delta^k$ is a retraction respecting the data above.
			\end{itemize}
		\item A morphism is a pair $(\alpha, f):(P,Y,r)\to (P',Y',r')$ where $\alpha:P\to P'$ is a morphism in $\Q(X^{\Delta^k})$ and $f:Y\to Y'$ is a morphism above respecting all the data. This is completely determined by $\alpha$ if such an $f$ exists.
		\item A morphism $(\alpha, f)$ is a cofibration of $\alpha$ is a cofibration in $\Q(X^{\Delta^k}).$
		\item A morphism $(\alpha, f)$ is a $x$-equivalence if $\alpha$ is a $w$-equivalence in $\Q(X^{\Delta^k})$ and $f$ sends every cell in $\ker \alpha$ into $X\times\Delta^k.$
		\item A morphism $(\alpha, f)$ is an $h$-equivalence if $f$ is a homotopy equivalence.
	\end{itemize}
Altogether, this defines a simplicial Waldhausen category.
\end{Definition}

There is a map
	$$x\E(X^{\Delta^\bullet})\to w\Q(X^{\Delta^\bullet})$$
given by forgetting about the cells. Igusa and Waldhausen showed \cite{IgusaWaldhausen}. 

\begin{Proposition}\label{Prop4.2} This map gives a weak equivalence 
	$$Kx\E(X^{\Delta^\bullet})\simeq Kw\Q(X^{\Delta^\bullet})\simeq Q(X_+).$$
\end{Proposition}

\begin{proof} This was originally proved in \cite{IgusaWaldhausen}. We reproduce the proof in the appendix \ref{homotopyofQ}.
\end{proof}

Furthermore there is a map 
	$$h\E(X^{\Delta^\bullet})\to h \R^{hf}(X^{\Delta^\bullet})$$
given by forgetting the graded posets. 

\begin{Proposition}\label{Prop4.3} This gives a weak equivalence
	$$Kh\E(X^{\Delta^\bullet})\simeq A(X^{\Delta^\bullet})\simeq A(X).$$
\end{Proposition}

\begin{proof} Again, this was proved in \cite{IgusaWaldhausen} and can be found in the appendix \ref{homotopyofA}.
\end{proof}

One can see that every $x$-equivalence of $\E(X^{\Delta^\bullet})$ is also an $h$-equivalence, and hence we get a map 
	$$Kx\E(X^{\Delta^\bullet})\to Kh\E(X^{\Delta^\bullet}).$$
We will use this map to compare the Becker-Gottlieb transfer and the Igusa-Klein transfer as maps into $A(X).$

\begin{Remark} The fiber of the above map can be identified as
	$$Kx\E(X^{\Delta^\bullet})^h\simeq Wh^{PL}(X).$$
Doing so was the original purpose of the Igusa-Waldhausen paper \cite{IgusaWaldhausen}. Since Waldhausen found an alternate proof in \cite{zbMATH03927168}, this paper was ultimately never published.
\end{Remark}

\subsection{A lift for the Becker-Gottlieb transfer}

In this section we work to compare $p^!_{IK}$ and $p^!$ from the previous sections. Much will be guided by the following homotopy commutative diagram (the disconnected part on the right indicates the homotopy type of every model in the corresponding rows, all horizontal maps are weak equivalences). As always consider a smooth bundle $E\to B$ with local system $\F\to E.$
\footnotesize
$$\xymatrix{
			&	Kw\K(\simE,\xi_\F)	&	Kx\E(\simE)\ar[l]_-{MMC_\xi}\ar[d]	&	\Omega\hocolim_n|\S_\bullet\P_\bullet(E\times I^n)|\ar[d]	& Q(E_+)\ar[d]^{\textnormal{Assembly}}\\
			&					&	Kh\E(\simE)\ar[r]				&	Kh\R^{hf}(\simE)								& A(E)\\
	|\S_\bullet|(B)|\ar[uur]^{p^!_{IK}}\ar[uurr]^{p^!_M}\ar[urr]^{p^!_{A,M}}\ar[urrr]_{p^!_A}\ar@/^8pc/[uurrr]^{p^!}
	}
$$

\normalsize

We already defined all the spaces involved as well as the transfers $p^!_{IK}$ and $p^!_A.$ The strategy to introduce and use the rest is as follows: First we define the Morse transfer $p^!_M:|\S_\bullet|\to Kx\E(\simE)$ and the map $MMC_\xi$ The transfer $p^!_{A,M}$ is simply the composition of $p^!_M$ with the inclusion $Kx\E(\simE)\to Kh\E(\simE).$ Recall that we have an explicit description of $p^!_A$ from Section \ref{smoothtorsion} which we now can compare explicitly to the composition of $p^!_A$ with the forgetful inclusion $Kh\E(\simE)\to \R^{hf}(\simE)$ (which is a weak equivalence as we will show in Appendix \ref{homotopyofA}). Lastly, $p^!$ and $p^!_{IK}$ are both lifts of $p^!_A$ and $p^!_{A,M},$ so in the end we leverage the section of the assembly map to lift our comparison. 

\begin{Definition} The Morse transfer
	$$p^!_{M}:|S_\bullet(B)|\to Kx\E(E^{\Delta^\bullet})$$
is given by sending a simplex $\sigma:\Delta^k\to B$ to the pair $(P,Y),$ where $P$ is the graded poset of the critical and twice the birthdeath points over $\sigma$ (as before this can be viewed as a set over $E^{\Delta^k}$) and 
	$$Y=E\times \Delta^k\sqcup Y'.$$
Here  $Y'\simeq\sigma^*E$ is viewed as a parametrized cell complex via the generalized Morse function. More precisely
	$$Y'=\bigsqcup_{p\in P}I^{\deg p}\times \Delta^k/\sim.$$
The equivalence relation does not only identify attachments of the boundary of cells, it also identifies whenever two critical points join together at a birth-death point. Altogether $Y$ gives a parametrized retractive space of $E\times \Delta^k$ where the retraction is given by inclusion of the Morse skeleton in the first component
	$$Y=Y\times \Delta^k\sqcup Y'\to Y\times \Delta^k.$$
\end{Definition}

There is a functor
	$$MMC_\xi:x\E(E^{\Delta^\bullet})\to w\K(E^{\Delta^\bullet},\xi)$$
Constructed in the following way: An object of $x\E(E^{\Delta^k})$ is a pair $(P,Y)$ where $P$ is a graded poset with an inclusion $P\times \Delta^k\to E$ and 
	$$Y=E\times \Delta^k\sqcup\bigsqcup_{p\in P}I^{\deg p}\times \Delta^k/\sim$$
can be viewed as a parametrized $\Delta^k$ family of relative cell complexes over $E$ with cells indexed and attached according to the poset order of $P.$ In particular every vertex $[l]\in \Delta^k$ gives a cell complex $Y(l)$ and every edge $[l,l']\subset \Delta^k$ gives a simple homotopy equivalence $Y(l)\to Y(l').$ Higher faces of $\Delta^k$ will give homotopies and higher homotopies between these simple homotopy equivalences. So we can form $MMC_\xi(P,Y)$ by setting 
	$$MMC_\xi(P,Y):=MMC((P,C)),$$
where $C$ is the $k$-tuple of  chain complexes given by $C(l)=C_*(Y(l),E;\xi)$ (where $1\leq l\leq k$) together with homotopy equivalences and higher homotopies between $Y(l)$ and $Y(l')$ given by the simple homotopy equivalences and higher homotopies from above. The functor $MMC$ is the multiple mapping cylinder. Notice that this can be done functorially.

We get a functor in $K$-theory and it follows directly that
	$$p^!_{IK}=MMC_\xi \circ p^!_{M}:|S_\bullet(B)|\to Kw\K(E^{\Delta^\bullet},\xi).$$

\begin{Proposition} The map	
	$$MMC_\xi:Kx\E(E^{\Delta^\bullet})\to Kw\K(E^{\Delta^\bullet},\xi)$$
is a weak equivalence.
\end{Proposition}

\begin{proof} Both the forgetful map
	$$x\E(E^{\Delta^\bullet})\to w\Q(E^{\Delta^\bullet})$$
and the composition with the forgetful map
	$$Kx\E(E^{\Delta^\bullet})\stackrel{MMC_\xi}\longrightarrow Kw\K(E^{\Delta^\bullet},\xi)\to w\Q(E^{\Delta^\bullet})$$
agree by inspection (they only care about the first component). Furthermore, both forgetful maps are homotopy equivalences.
\end{proof}

\begin{Corollary} The maps $p^!_{IK}$ and $p^!_{M}$ have the same homotopy type viewed as maps 
	$$B\to Q(E_+).$$
\end{Corollary}

Consequently, it is enough to show that $p^!_{M}$ has the homotopy type of the Becker-Gottlieb transfer.

\subsection{Determining the homotopy type}

Our goal now is to show that $p^!$ and $p^!_{M}$ have the same homotopy type. Recall that $p^!$ was given as an explicit lift to 	
	$$p^!_A:|S_\bullet(B)|\to A(E^{\Delta^\bullet})$$
defined by sending $\sigma:\Delta^k\to B$ to the retractive space $E\times \Delta^k\sqcup \sigma^*E\to E\times \Delta^k.$ 

On the other hand $p^!_{M}$ was explicitly constructed as a geometric realization
	$$p^!_{M}:|S_\bullet(B)|\to Kx\E(E^{\Delta^\bullet}).$$
Furthermore there is the inclusion
	$$Kx\E(E^{\Delta^\bullet})\to Kh\E(E^{\Delta^\bullet})\simeq A(E^{\Delta^\bullet}).$$
We denote the composition of this with $p^!_M$ by $p^!_{A,M}.$ Instead of comparing maps on the level of $Q(E_+)$ we will compare $p^!_A$ and $p^!_{A,M}$ on the level of $A(E).$ First we need

\begin{Lemma} The  map
	$$Kx\E(E^{\Delta^\bullet})\to Kh\E(E^{\Delta^\bullet})$$
has the homotopy type of the assembly map
	$$a:Q(E_+)\to A(E)$$
\end{Lemma}

\begin{proof} All of these maps can be viewed as maps of ring spectra. If $E=*$ there is only the initial map $Q(S^0)\to A(S^0)$ since the sphere spectrum $Q(S^0)$ is the initial object in that category. The lemma now follow from universality arguments for any $E.$
\end{proof}

\begin{Theorem} The transfers $p^!$ and $p^!_{M}$ have the same homotopy type as maps
	$$B\to Q(E_+)$$
\end{Theorem}

\begin{proof} First notice that there is a homotopy equivalence between $a\circ p^!=p^!_A$ and $a\circ p^!_{M}=p^!_{A,M}$ both seen as maps
	$$|S_\bullet(B)|\to Kh\R^{hf}(\simE)\simeq A(E^{\Delta^\bullet})$$
(for $p^!_{A,M}$ we need to compose with the inclusion $Kh\E(\simE)\to Kh\R^{hf}(\simE)$): On the simplex $\sigma:\Delta^k\to B$ the equivalence is given by the natural transformation coming from including the Morse-skeleton $Y'$ into $\sigma^*(E)$ giving an homotopy equivalence inclusion
	$$(a\circ p^!_{M})(\sigma)=E\times \Delta^k\sqcup Y'\into E\times \Delta^k\sqcup \sigma^* E=p^!_A(\sigma)$$
Furthermore, Waldhausen showed  \cite{zbMATH03958257} that there is a homotopy right inverse to $a$ given by the trace map $\tr:A(E)\to Q(E_+),$ so $\tr\circ a\sim \id$ (this map gives the splitting of the fibration sequence $Wh^{PL}(E)\to Q(E_+)\to A(E)$). So finally we have
	$$p^!\sim\tr\circ a\circ p^!\sim\tr\circ a \circ p^!_{M}\sim p^!_{M}.$$
\end{proof}

\begin{Corollary} So we established
	$$p^!\sim p^!_{IK}.$$
\end{Corollary}

\subsection{Comparing the lifts}

So far we established that the two different transfer maps $p^!, p^!_{IK}:|S_\bullet(B)|\to Q(E_+)$ agree. But to prove Theorem \ref{mainthm}, we are interested in comparing their lifts $\tau_{sm}, \tau_{IK}:|S_\bullet(B)|\to Wh_\xi(E),$ which are uniquely determined by their underlying maps $p^!$ and $p^!_{IK}$ together with a homotopy $H_{sm}:\lambda\circ p^!\to \const_0$ and $H_{IK}:\lambda\circ p^!_{IK}\to \const_0$ where $\lambda:Q(E_+)\to K(\C)$ is the assembly map followed by linearization. 

To prove that $\tau_{sm}$ and $\tau_{IK}$ are homotopic, we need to provide a homotopy of homotopies $\H:H_{IK}\circ H\to H_{sm},$ where $H:p^!\to p^!_{IK}$ is the homotopy found above. So $\H$ is a homotopy of homotopies of maps $Q(E_+)\to K(\C).$ One can view it as the inside of the triangle-diagram
	$$\xymatrix{
		\lambda\circ p^!\ar[dr]^{H_{sm}}\ar[rr]^{\lambda\circ H}	&&	\lambda\circ p^!_{IK}\ar[dl]_{H_{IK}}\\
			& \const_0,
		}
	$$
where all corners are maps $Q(E_+)\to K(\C).$

To be precise we will be working with $p^!_{M}:|S_\bullet(B)|\to Kx\E(E)$ instead of $p^!_{IK}.$ We need to be careful because the lift of $p^!$ (and $p^!_{IK}$) is explicitly given. However, by theorem \ref{fibthm}, it corresponds to the contracting homotopy $\H_{M}$ of the composition
	$$|S_\bullet(B)|\to Kx\E(E)\to K(\C)$$
given by the natural transformation to the final functor $\const_0:S_\bullet(B)\to K(\C).$

The homotopy $H_{sm}$ come from theorem \ref{contraction} and if the homology of the fiber is acyclic it is just given by the same final map \cite{twistedsmooth}. So it is clear that the diagram above commutes on the nose. This completes the proof of Theorem \ref{mainthm}.

\begin{Remark} So far this only works if the fiber $F$ of $E\to B$ is acyclic with respect to the local system $\xi.$ If it is not, the definition of Igusa-Klein torsion asks that the fundamental group $\pi_1(B)$ act trivially on the homology $H_*(F;\xi)$ (the smooth torsion only asks for this action to be unipotent). In this case, to define the smooth torsion we subtract the constant functor $H_*(F;\xi)$ from the construction in the loop space structure of $K(\C)$ (compare Definition \ref{reducedtorsion}). For the Igusa-Klein torsion, one forms a certain mapping cone as done in \cite{Igusa2}. Both amount to the same outcome and both torsions are still going to be equivalent.
\end{Remark}

\section{Appendix: Two Proofs}

\renewcommand{\Z}{\mathcal{Z}}

We present the proofs for Propositions \ref{Prop4.2} and \ref{Prop4.3}. These already appeared in \cite{IgusaWaldhausen}, but remained unpublished and not publicly accessible. We merely reproduce the results.

\subsection{The Homotopy Type of $Kx\E(X^{\Delta^\bullet})$}
\label{homotopyofQ}

Let $X$ be a topological space. We aim to prove
	$$Kx\E(\X)\simeq \Q(X_\bullet).$$
	
We will need an auxiliary category $\D.$

\begin{Definition} The simplicial category $\D(\X)$ is the same as the category $\E(\X)$ without identified expansion pairs. Explicitly we define  $\D(\X[k])$ as follows:
	\begin{itemize} 	
		\item An object in degree $k$ of $\D(\X[k])$ is a triple $(P,Y,r),$ where
			\begin{itemize}
				\item $P$ is a graded poset over $X^{\Delta^k}$ (without identified expansion pairs).
				\item $Y$ is a $k$-parameter family of relative cell complexes with cells indexed by $P.$ In particular
					$$Y=X\times \Delta^k\sqcup \bigsqcup_{p\in P} I^{\deg{p}}\times \Delta^k/\sim,$$
				where no cell is attached to a cell of equal or higher order. 
				\item $r:Y\to X\times \Delta^k$ is a retraction respecting the data above.
			\end{itemize}
		\item A morphism is a pair $(\alpha, f):(P,Y,r)\to (P',Y',r')$ where 
			\begin{itemize}
				\item $\alpha:P\to P'$ is a pointed set map that is closed as a poset map
				\item $f:Y\to Y'$ is a morphism above it respecting all the data such that 
				\item if $A$ is closed subset of $P$ then $f(Y^A)\subset (Y')^{\alpha(P)}$ where $Y^A\subset Y$ is the set of all elements over $A.$
			\end{itemize}
		\item A morphism $(\alpha, f)$ is a cofibration if $\alpha$ is an order preserving monomorphism (making $f$ an embedding of a parametrized subcomplex)
		\item A morphism $(\alpha, f)$ is a weak equivalence if $\alpha$ is a bijection (and thus $f$ is a parametrized cellular homeomorphism).
	\end{itemize}
Altogether, this defines a simplicial Waldhausen category. Let $\D^1(\X)$ be the full subcategory with cofibrations of $\D(\X)$ of objects $(Y,P,r)$ where $P$ only has the trivial ordering. Intuitively, this means that in $\D^1$ all the cells are attached at once.

\end{Definition}

\begin{Remark} A morphism $(f,\alpha)$ in $\D(\X)$ is completely determined by $\alpha$ if it exists.
\end{Remark}

Again let $\Q_0(\X)$ be the simplicial categories of finite sets (neither graded nor ordered) over $\X.$ Then recall from Remark \ref{ordering}
	$$K\Q(\X)\simeq K\Q_0(\X)\simeq \Omega^\infty\Sigma^\infty X_+.$$
	
Notice that the weak equivalences of $\D^1(\X)$ are exactly the isomorphisms because everything has the trivial ordering. Furthermore every object in $\D^1(\X[k])$ splits uniquely as a sum of objects with each only having cells in one given dimension.
	
\begin{Lemma} Let $n\in \mathbb{N}$ and let $\D^1_n(\X)$ be the subcategory of $\D^1(\X)$ with only cells of degree $n.$ Then we have
	$$\D^1_n(\X)\simeq \Q_0(\X).$$
\end{Lemma}

\begin{proof}  Let $f:\D^1_n(\X[k])\to \Q_0(\X[k])$ be the forgetful functor with 
	$$f(Y,P,r):=(P, \gamma),$$
where $\gamma:P\to \X[k]$ is given by the attachments of the basepoints of the cells that make $Y.$ Let $j:\Q_0(\X[k])\to \D^1_n(\X[k])$ be the functor that is given by 
	$$j(P,\gamma):=(Y,P,r),$$
where 
	$$Y=X\times \Delta^k \sqcup\bigsqcup_{p\in P}I^n\times \Delta^k/\sim$$
with all cells attached at their basepoint via $\gamma:P\to \X[k].$ The retraction $r$ is simply given by mapping the cells to their basepoint.

Clearly, the composition $f\circ j:\Q_0(\X[k])\to \Q_0(\X[k])$ is the identity. On the other hand, the composition $j\circ f:\D^1_n(\X)\to\D^1_n(\X)$ is given by contracting all attachment maps to attachments at the basepoints of the cells. A homotopy $j\circ f\sim \id$ comes from the functors
	$$H:\D^1_n(\X[k])\times \Delta([k],[1])\to \D^1_n(\X[k])$$
given by sending $((Y,P,r),\alpha)\mapsto (Y_\alpha,P,r_\alpha).$ Here $Y_\alpha$ has the same cells as $Y$ with different attachment maps: Let $\eta:\partial I^n\times \Delta^k\to X$ be the attachment map of a cell of $Y$ (Notice $\eta=r_{|\partial I^n}$), then the new attachment map is $\eta_\alpha:\partial I^n\times \Delta^k\to X$ with $\eta_\alpha(s,t)=r(\alpha_*(t)s,t),$ where $\alpha_*:\Delta^k\to \Delta^1=I$ is the induced map. The retraction $r_\alpha:I^n\times \Delta^k\to X$ is given by  $r_\alpha(s,t)=r(\alpha_*(t)s,t).$

\end{proof}

\begin{Proposition}\label{Prop5.3} We have 
	$$i\S_\bullet\D^1(\X)\simeq i\S_\bullet \D(\X)\simeq w\S_\bullet\D(\X).$$
\end{Proposition}

\begin{proof} It will be enough to show $i\S_\bullet \D^1(\X[k])\simeq i\S_\bullet \D(\X[k])$ and that $w\S_\bullet \D(\X[k])$ is homotopy equivalent to a simplicial subcategory of $i\S_\bullet\D(\X[k])$ which contains $i\S_\bullet\D^1(\X[k]).$ We will show the latter first.

Let $\bar\D(\X[k])$ be the subcategory of $\D(\X[k])$ of objects $(Y,P,r)$ where the partial ordering on $P$ is as minimal as possible. While this is not a subcategory with cofibrations as it does not have all push-outs, it does have quotients and we can form $i\S_\bullet\bar\D(\X[k])\subset i\S_\bullet \D(\X[k])$ and this subcategory contains $i\S_\bullet \D^1(\X[k])$ as demanded. Let 
	$$g:w\S_\bullet\D(\X[k])\to i\S_\bullet\bar\D(\X[k])$$
be the functor given by sending $(Y,P,r)$ to $(Y,P',r)$ where $P=P'$ as sets and $P'$ has the minimally necessary ordering. Let 
	$$j:i\S_\bullet\bar\D(\X[k])\to w\S_\bullet\D(\X[k])$$
be given by the inclusion. Then we have $gj=\id$ and $jg\sim \id$ as weak equivalences in $w\D(\X[k])$ are given by set-bijections on the posets.

It now suffices to show that $i\S_\bullet\D^1(\X[k])\simeq i\S_\bullet\D(\X[k]).$ Let $\D^n(\X[k])$ be the full subcategory of $\D(\X[k])$ in which cells are attached in no more than $n$ layers. We will show inductively $i\S_\bullet \D^n(\X[k])\simeq i\S_\bullet \D^{n+1}(\X[k]).$

Let $\Z$ be the Waldhausen category with objects being pairs $((Y,P,r),z)$ where $(Y,P,r)\in \D^{n+1}(\X[k])$ and $z:P\to \{0,\frac 1 2 ,1\}$ is a ``height function'' with
	
	\begin{itemize}
		\item Every element of $z^{-1}\{0,\frac 1 2 \}$ is minimal and
		\item the poset $z^{-1}\{\frac 1 2 ,1\}$ does not contain any $(n+1)$-chains.
	\end{itemize}
  The morphisms are the morphisms $(f,\alpha):(Y,P,r)\to (Y',P',r')$ such that $\alpha$ takes $z^{-1}\{0\}$ into $(z'^{-1}\{0\})_+$ and  $z^{-1}\{1\}$ into $(z'^{-1}\{0,1\})_+.$ A cofibration is a height perserving cofibration in $\D(\X[k])$ and a weak equivalence $(f,\alpha):(Y,P,r)\to (Y',P',r')$ induces an isomorphism $Y\to Y'$ so that $\alpha$ sends $z^{-1}\{1\}$ into $z'^{-1}\{1\}.$
  
Let $\Z_0$ be be the full Waldhausen subcategory of $\Z$ with objects from $\D^1(\X[k])$ and let $\E$ be the full Waldhausen subcategory of $\Z$ given by $z^{-1}\{\frac 1 2 \}=\emptyset$ and let $\E_0=\E\cap \Z_0.$ Then $\E$ is exactly equivalent to the category of cofibrations $A\into B\onto B/A$ with $A\in \D^1(\X[k]),$ $B\in \D^{n+1}(\X[k]),$ and $B/A\in \D^n(\X[k])$ and $\E_0$ is exactly the category of cofibration sequences in $\D^1(\X[k]).$ By  the additivity theorem \cite{zbMATH03927168} and induction we have
	$$i\S_\bullet \E\simeq i\S_\bullet(\D^1(\X[k])\times \D^n(\X[k]))\simeq i\S_\bullet (\D^1(\X[k])\times \D^1(\X[k]))\simeq i\S_\bullet \E_0.$$
	
We will consider the following map of fibration sequences
	$$\xymatrix{
		i\S_\bullet\E_0\ar[r]\ar[d]	&	w\S_\bullet \Z_0\ar[d]\ar[r]	&	w\S_\bullet\S_\bullet(\E_0\to\Z_0)\ar[d]\\
		i\S_\bullet\E\ar[r]	&	w\S_\bullet \Z\ar[r]	&	w\S_\bullet\S_\bullet(\E\to\Z).
		}
	$$
The left hand map is a weak equivalence. If we can show that the right hand map also is one, we will get $w\S_\bullet \Z_0\simeq w\S_\bullet \Z$ which in light of Lemma \ref{Lemma5.4} proves the proposition. We have $w\S_\bullet\S_\bullet(\E\to\Z)\simeq w\S_\bullet \Z k\Z$ and $w\S_\bullet\S_\bullet(\E_0\to\Z_0)\simeq w\S_\bullet \Z_0 k\Z_0.$

Let $j:\Z_0\to \Z$ be the inclusion functor and $q:\Z\to\Z_0$ be the functor that changes the attachment maps $\psi:\partial I^n\times \Delta^k\to X\times \Delta^k$ to $r\circ \psi$ where $r:Y\to X$ is the retraction, thereby attaching cells to the base directly. Clearly, we have $qj=\id_{\Z_0}.$ Consider the functor $h:\Z\to \Z_0$ given by sending $(Y,P,r)$ to only its minimal cells. There clearly are natural transformations given by inclusions
	$$\id_\Z\ot jh\to jq.$$
While the functor $h$ is not exact, it still gives a morphism of bicategories
	$$h:\S_n\Z k\Z\to w\S_n \Z_0 k\Z_0$$
and the tranformations above give homotopies between $jq$ and $\id$ considered as functors $h:\S_n\Z k\Z\to w\S_n \Z k\Z.$

\end{proof}

\begin{Lemma}\label{Lemma5.4} The forgetful functors $\Z\to \D^{n+1}(\X[k])$ and $\Z_0\to \D^1(\X[k])$ induce weak equivalences $w\S_\bullet \Z\simeq i\S_\bullet\D^{n+1}(\X[k])$ and $w\S_\bullet \Z_0\simeq i\S_\bullet\D^{1}(\X[k]).$
\end{Lemma}

\begin{proof} We will use Qillen's Theorem A \cite{MR2655184} to show that $f_m:w\S_m\Z\to i\D^{n+1}(\X[k])$ is a weak equivalence by showing that $f_m/P$ is contractible for every $P=((Y_1,P_1,r_1)\into\ldots\into (Y_m,P_m,r_m)).$ We will do so by providing an initial object, first in the case $m=1.$ In this case $P$ is a single complex with poset $P_1.$ We can give a height function $z_0:P_1\to \{0,\frac 1 2 ,1\}$ by
	$$z_0(x):=\left\{ \begin{array}{ll}
			1 & \textnormal{if $x$ is not minimal}\\
			0 & \textnormal{if $x$ is minimal and belongs to a chain of length $n+1$}\\
			\frac 1 2  & \textnormal{else}
			\end{array}
			\right.
	$$
This provides the initial object. The case $m>1$ and the case of $Z_0$ are similar.
\end{proof}

Finally we can prove:

\begin{Theorem} The simplicial forgetful functor $\E(\X)\to \Q(\X)$ induces a weak equivalence
	$$Kx\E(\X)\simeq Kw\Q(\X)\simeq \Omega^\infty\Sigma^\infty X_+.$$
\end{Theorem}

\begin{proof} For $i\leq j$ let $\E^j_i(\X)$ be the subcategory of cell complexes with cells only with degrees between $i$ and $j.$ Since expansion pairs require cells in different dimensions we have $\E^i_i(\X)\cong \D_i(\X),$ where $\D_i(\X)$ only  contains cells in dimension $i.$ From the discussion above we have
	$$x\S_\bullet\E^i_i(\X)\simeq w\S_\bullet\D_i(\X)\simeq i\S_\bullet\D^1_i(\X)\simeq i\S_\bullet \Q_0(\X)\simeq w\S\Q(\X).$$

We also have
	$$x\S_\bullet \E(\X)\simeq \colim_j x\S_\bullet \E_0^j(\X).$$
So it suffices to show the following lemma:

\begin{Lemma} The inclusion induces a homotopy equivalence
	$$x\S_\bullet\E^i_i(\X)\simeq x\S_\bullet\E^j_i(\X)$$
for all $0\leq i\leq j.$
\end{Lemma}

\begin{proof} Let $\B\subset \E_i^j(\X[k])$ be the subcategory of all $(Y,P,r)$ such that $P$ consists only of expansion pairs. Furthermore let $k\E_i^j(\X[k])$ be the subcategory of all cofibrations in $\E^i_j(\X[k])$ with quotients in $\B.$ Let $v\E_i^j(\X[k])$ be the subcategory of $x\E_i^j(\X[k])$ of all collapsing maps $(f,\alpha)$ such that $\ker\alpha$ is a union of expansion pairs and $\alpha$ induces an isomorphism of graded posets when restricted to $\textnormal{coim}\, \alpha.$ Notice that the $v$-weak equivalences are canonical left-inverses for the $k$-weak equivalences. This can be used to show that the $v$-equivalences do in fact form a category of generalized equivalences. 

Let $u\E_i^j(\X[k])$ be the subcategory of $x\E_i^j(\X[k])$ of all $(f,\alpha)$ where $\alpha$ is a bijection. This again is a category of weak equivalences. We conclude that there is a homotopy fiber sequence
	$$i\S_\bullet \B\to u\S_\bullet \E_i^j(\X[k])\to uv\S_\bullet\E_i^j(\X[k]),$$
where the latter is the simplicial bicategory given by $u\S_n\E_i^j(\X[k])v\S_n\E_i^j(\X[k])$ in degree $n.$

We continue to identify the terms of this sequence. First of all, we see that $\B$ is equivalent as a category with cofibrations to $(\D^1)_{i+1}^j(\X[k]).$ Consequently we have
	$$i\S_\bullet\B\simeq i\S_\bullet(\D^1)_{i+1}^j(\X[k]).$$
	
Let $\varepsilon:\E_i^j(\X[k])\to \D_i^j(\X[k])$ be the functor that unpairs all expansion pairs and $j$ the inclusion. We get $\varepsilon\circ  j=\id$ and there is a natural $u$-equivalence $j\varepsilon \simeq \id,$ so overall we learn
	$$u\S_\bullet\E_i^j(\X[k])\simeq w\S_\bullet \D_i^j(\X[k]).$$
Furthermore by Proposition \ref{Prop5.3} above we have
	$$w\S_\bullet\D_i^j(\X[k])\simeq i\S_\bullet(\D^1)_i^j(\X[k])\simeq  i\S_\bullet(\D^1)_{i+1}^j(\X[k])\times  i\S_\bullet(\D^1)_i^i(\X[k]).$$

From the homotopy fiber sequence above we can conclude
	$$uv\S_\bullet \E_i^j(\X[k])\simeq i\S_\bullet (D^1)_i^i(\X[k])\simeq x\S_\bullet \E_i^i(\X[k]).$$
So it remains to show
	$$uv\S_\bullet \E_i^j(\X[k])\simeq x\S_\bullet \E_i^j(\X[k]).$$

Notice that every $x$-equivalence in $\E_i^j(\X[k])$ splits naturally as the composition of a $u$- and $v$-equivalence. This carries through to give a splitting
	$$x\S_n\E_i^j(\X[k])\simeq u\S_n\E_i^j(\X[k])v\S_n\E_i^j(\X[k])\simeq uv\S_n\E_i^j(\X[k]).$$

\end{proof}
	
\end{proof}

\subsection{The Homotopy Type of $\K h\E(X^{\Delta^\bullet})$ }
\label{homotopyofA}

We move on to show
	$$Kh\E(\X)\simeq A(X).$$
We still follow \cite{IgusaWaldhausen}. Again, we will use a supplemental category. However, before defining it, we need to lay some ground work defining mapping cylinders in various categories as we will be using Waldhausen's approximation theorem \cite{zbMATH03927168}.

\begin{Definition} Given a morphism $\alpha:P\to P'$ in $\Q(\X[k])$ we define its mapping cylinder $T(\alpha):=P\vee P'\vee \Sigma\delta P$ where $\delta P$ deletes all expansion pairs and $\Sigma$ increases every degree by 1. Furthermore we have $z\leq \sigma x$ for $z\in P',$ $\sigma x\in \Sigma\delta P$ iff $z\leq \alpha(x).$ There are obvious maps $P\vee P'\into T(\alpha)$ and $T(\alpha)\to P'.$
\end{Definition}

\begin{Definition} Let $(f,\alpha):(Y,P,r)\to (Y',P',r')$ be a morphism in $\E(\X[k]),$ then its mapping cylinder is given by $(T(f), T(\alpha),r'')$ where $T(f)$ is the  topological reduced mapping cylinder  and $T(\alpha)$ is the mapping cylinder from above. The retraction $r''$ is given canoincally.
\end{Definition}

One can verify that these define proper cylinder functors on the Waldhausen categories $\Q(\X[k])$ and $\E(\X[k]).$

\begin{Definition} Let $X$ be a space, we define the Waldhausen category $\M(\X[k])$ in the following way:
	\begin{itemize}
		\item The objects of $\M(\X[k])$ are $(Y,P,r)$ -- the same as for $\E(X[k]).$
		\item A morphism is a pair $(f,\alpha):(Y,P,r)\to (Y',P',r')$ where 
			\begin{itemize}
				\item $\alpha:\Lambda P\to \Lambda P'$ is a $\vee$-preserving map and $\Lambda P$ is the set of closed subsets of $P$
				\item $f:Y\to Y'$ is a continuous map fixing $X\times \Delta^k$ and commuting with the retraction such that
				\item $f$ maps $Y^A$ into $Y^{\alpha(A)}$
			\end{itemize}
		\item A map $(f,\alpha)$ is a cofibration if 
			\begin{itemize}
				\item $\alpha$ is induced via the inclusion $P\to \Lambda P$ ($x\mapsto \bar x:= \{y|y\leq x\}$) by a cofibration $P\to P'$ in $\Q(\X[k])$
				\item $f$ is a homeomorphism of $Y^A$ onto $Y^{\alpha (A)}$ for all $A\in \Lambda P.$
			\end{itemize}
		\item A weak equivalence $(f,\alpha)$ is an $h$-equivalence meaning that $f$ is a homotopy equivalence.
	\end{itemize}
\end{Definition}

\begin{Remark} The two main differences between $\E(\X[k])$ and $\M(\X[k])$ are that the latter has no expansion pairs and more morphisms, as every morphism $(f,\alpha)$ in $\E(\X[k])$ is completely determined -- if existent -- by $\alpha.$
\end{Remark}

\begin{Definition} Let $(f,\alpha):(Y,P,r)\to (Y',P',r')$ be a morphism in $\M(\X[k]).$ It's mapping cylinder is given by $(T(f),T(\alpha),r'')$ where $T(\alpha)=P\vee P'\vee \Sigma P.$ The projection $T(\alpha)\to P'$ is given by sending a closed subset $A\vee B\vee \Sigma C$ to $\alpha(A)\cup B.$ Again, $T(f)$ is the mapping cylinder.
\end{Definition}

This defines a cylinder functor on $\M(\X[k]).$

\begin{Lemma} We have
	$$h\S_\bullet\M(\X[0])\simeq h\S_\bullet \R_f(X)$$
and hence $Kh\M(\X[0])\simeq A(X).$
\end{Lemma}

\begin{proof} Let $\M(\X[0])^{CW}$ be the subcategory of $\M(\X[0]),$ where all cells are attached in order of degree (hence all objects are CW-complexes). This becomes a Waldhausen subcategory with cylinder functor.

We will use Waldhausen's approximation theorem to show that both inclusions $h\S_\bullet\M(\X[0])^{CW}\to h\S_\bullet\R_f(X)$   and $h\S_\bullet\M(\X[0])^{CW}\to h\S_\bullet\M(\X[0])$ are weak equivalences.  For the former this is straight forward. 

For the latter let $(f,\alpha):(Y,P,r)\to (Y',P',r')$ be any morphism in $\M(\X[0])$ such that $Y$ is a CW-complex. By CW-approximation there is a weak equivalence $f':(Y'',r'')\to (Y',r')$ where $Y''$ is a CW complex. Let $P$ be the graded poset of cells of $Y''.$ We define a morphism $(f',\alpha):(Y'',P'',r'')\to (Y',P',r')$ by setting $\alpha(A)=P'$ for all $A\in \Lambda P''.$ 

Cellular approximation gives a homotopy approximation $f_h:(Y,r)\to (Y'',r'')$ to $f:Y\to Y'$ and we enrich it to $(f_h,\alpha_h):(Y,P,r)\to (Y'',P'',r'')$ by setting $\alpha_h(\bar x):=\{y|\deg y\leq \deg x\}.$ Now one can see that $(f,\alpha)$ factors as
	$$(Y,P,r)\into (T(f_h),T(\alpha_h),r)\to (Y',P',r').$$
\end{proof}

\begin{Lemma} The degenerate inclusion induces a weak equivalence
	$$h\S_\bullet \M(\X[0])\to h\S_\bullet\M(\X).$$
\end{Lemma}

\begin{proof} It is enough to show 
	$$h\S_\bullet \M(\X[0])\to h\S_\bullet\M(\X[k])$$
for all $k.$ Call the degeneracy operator $S:\M(\X[0])\to \M(\X[k]).$ We will use the approximation theorem again and show that $S$ satisfies the approximation property.

Let $(f,\alpha):(Y,P,r)\to (Z,Q,s)$ be a morphism in $\M(\X[k])$ where $(Y,P,r)$ is degenerate. Denote the restriction to the first vertex by $(f_0,\alpha_0).$ Let $(T(f_0), T(\alpha_0))$ be the mapping cone of $(f_0,\alpha_0).$ Then $(f,\alpha)$ factors as 
	$$(Y_0,P_0,r_0)\into (T(f_0), T(\alpha_0),r)\stackrel{\simeq}{\longrightarrow} (Z_0,Q_0,s_0).$$
Recall that $Y=S(Y_0)=Y_0\times \Delta^k$ and $P=P_0.$ Now we can cellularly expand the construction above to get a homotopy equivalence 
	$$h:T(f\times \Delta^k)\to Z.$$
This is rooted in the fact that any $Z^B\to \Delta^k$ is a Serre fibration. Further notice that $(T(f\times \Delta^k),T(\alpha),r)\in S(\M(\X[k]))$ and this concludes the proof that $S$ has the approximation property. 
\end{proof}

\begin{Lemma} The simplicial forgetful functor induces a homotopy equivalence
	$$h\S_\bullet \E(\X[k])\to h\S_\bullet\M(\X[k]).$$
\end{Lemma}

\begin{proof} Denote the forgetful functor by $\epsilon_k.$ We will again show that it has the approximation property. Let $\epsilon_k:(Y,P,r)\to (Z,Q,s)$ be a morphism in $\M(\X[k]).$ Let $A$ be the closed subset of $T(\alpha)$ given by deleting all expansion pairs from $\Sigma P.$ Then $T(f)^A\simeq T(f)\simeq Z$ and the cofibration 
	$$\epsilon_k(Y,P,r)\into (T(f)^A,A,r)$$
lifts to $\E(\X[k]).$
\end{proof}

Altogether we have shown

\begin{Theorem} There is a weak equivalence
	$$Kh\E(\X)\simeq A(X).$$
\end{Theorem}

\bibliography{mybib}{}
\bibliographystyle{plain}

\end{document}